\documentclass{amsproc}
\usepackage{xcolor}
\usepackage{amssymb}
\newtheorem{theorem}{Theorem}[section]
\newtheorem{lemma}[theorem]{Lemma}
\newtheorem{op}[theorem]{Open}

\theoremstyle{definition}
\newtheorem{definition}[theorem]{Definition}
\newtheorem{example}[theorem]{Example}

\theoremstyle{remark}
\newtheorem{remark}[theorem]{Remark}
\newtheorem{note}[theorem]{Note}
\numberwithin{equation}{section}
\usepackage{epsfig} 
\usepackage{epstopdf} 

\newcommand{\te}{\text}

\newcommand{\ep}{\epsilon}

\begin{document}

\title[Some results on Lower Assouad and quantization dimensions]{Some results on Lower Assouad and quantization dimensions}



\author{Saurabh Verma}
\address{Department of Applied Sciences, IIIT Allahabad, Prayagraj, India 211015}
\email{saurabhverma@iiita.ac.in}
\thanks{The first author is supported by the SEED Grant Project of IIIT Allahabad, India. The second and third authors are financially supported by the Ministry of Education, India.}

\author{Ekta Agrawal}
\address{Department of Applied Sciences, IIIT Allahabad, Prayagraj, India 211015}
%
\email{ekta.agrawal5346@gmail.com}

\author{Shivam Dubey}
\address{Department of Applied Sciences, IIIT Allahabad, Prayagraj, India 211015}
%
\email{rss2022509@iiita.ac.in}

\date{January 1, 1994 and, in revised form, June 22, 1994.}



\begin{abstract}
In this paper, we first show that the collection of all subsets of \( \mathbb{R} \) having lower dimension \( \gamma \in [0,1] \) is dense in \( \Pi(\mathbb{R}) \), the space of compact subsets of \( \mathbb{R} \). Furthermore, we show that the set of Borel probability measures with lower dimension \( \beta \in [0, m] \) is dense in \( \Omega(\mathbb{R}^m) \), the space of Borel probability measures on \( \mathbb{R}^m \). We also prove that the quantization and the lower dimension of a measure \( \vartheta \) coincide with those of the convolution of \( \vartheta \) with a finite combination of Dirac measures. In the end, we compute the lower dimension of the invariant measure associated with the product IFS.

\end{abstract}

\maketitle

\section{Introduction}
The dimension analysis of any set and measure in $\mathbb{R}^m$ is a fundamental part of fractal geometry. Several dimensions are introduced in the literature, such as Hausdorff dimension $\dim_\mathcal{H}(\cdot),$ box dimension $\dim_B(\cdot),$ packing dimension $\dim_P(\cdot),$ quantization dimension $D_r(\cdot),$ Assouad dimension $\dim_A(\cdot),~L^q$ dimension $D(\cdot,q),$ and lower Assouad dimension $\dim_L(\cdot)$ (see \cite{Fal, Fraser, KNZ1, Mat3, Sh1, Zador}), while some gave insight at the global level, others at the local level. In general, determining the dimension of a fractal set is a challenging task; To address this, certain separation conditions are introduced in the literature, such as the strong separation condition (SSC), open set condition (OSC), strong open set condition (SOSC), weak separation property (WSP) and exponential separation condition (ESC), see \cite{Fraser, Schief, VP1}. Hochman \cite{MHOCHMAN} computed the Hausdorff dimension for self-similar sets and measures with algebraic contraction and translation on $\mathbb{R}$ under ESC, contributing significantly toward addressing Simon's dimension drop conjecture (see \cite{Simon}). Shmerkin \cite[Theorem 6.6]{Sh1} computed the $L^q$ dimension of self-similar measures in $\mathbb{R}$ under ESC. Several researchers attempted to introduce ESC for the general class of IFS, in a way B\'ar\'any and M. Verma \cite{BV1} studied weak ESC for self-similar IFS, ESC for graph-directed self-similar IFS, and ESC for common fixed point systems. Recently, Verma et al. \cite{VAM} defined modified ESC for the general class of IFS, including self-affine, bi-Lipschitz, and self-conformal IFS. Some recent work on the estimation of the dimension of inhomogeneous attractors, graphs of continuous functions, and self-similar sets can be explored in \cite{B, Bayart, Bed, SS3, DV1, DV2, DV3, PPV1, MV1}. The present article deals with the lower Assouad dimension and the quantization dimension.\\
Larman \cite{Larman}, and later K\"aenm\"aki and Lehrb\"ack \cite{AL} introduced the lower Assouad dimension, which is basically dual to the Assouad dimension and discloses the local behavior. It is noted that for the subset of $\mathbb{R}^m$ that contains isolated points, its lower dimension is zero. Fraser \cite[Example 2.5]{Fraser12} presented an open set in $\mathbb{R}$ with a lower dimension equal to zero, showing that it does not satisfy the open set property, while the Hausdorff, packing, and Assouad dimensions satisfy it. Chen et al. \cite[Theorem 1]{Chen} showed the accessible value of the lower dimension of the subsets in a doubling metric space. Generally, if $F$ is a totally bounded subset of $\mathbb{R}^m,$ then
$$\dim_L(F)\le \underline{\dim}_B(F)\le\overline{\dim}_B(F)\le\dim_A (F),$$
where $\underline{\dim}_B$ and $\overline{\dim}_B$ represent the lower and upper box dimensions, respectively. Larman \cite{Larman} showed that if $F$ is a compact subset of $\mathbb{R}^m,$ then 
$$\dim_L(F)\le \dim_{\mathcal{H}}(F)\le\dim_P(F)\le\dim_A(F).$$
If $F$ is compact and Ahlfors regular, then $\dim_L(F)=\dim_\mathcal{H}(F)=\dim_A(F)$ (see \cite[Theorem 6.4.1]{Fraser}). 
K\"aenm\"aki et al. \cite{AJM} proved that if a measure $\vartheta$ is doubling and fully supported on a closed set $F,$ then 
$$\dim_L(\vartheta)\le\dim_L(F)\le\dim_A(F)\le\dim_A(\vartheta).$$
Graf and Luschgy \cite{GL1} showed that for any probability measure $\vartheta$ on $\mathbb{R}^m,$ we have $\dim_\mathcal{H}(\vartheta)\le D_r(\vartheta),$ for all $r\ge 1.$ As a starting point, researchers deal with the simplest case: for a given class of self-similar measures, whether the quantization dimension coincides with the Hausdorff dimension of measure under some appropriate separation condition. Graf and Luschgy \cite[Theorem 5.11]{GL3} proved that it is true under OSC. Kesseböhmer and Zhu \cite{KZ2} established a formula for the quantization dimension of a special class of self-affine measures supported on Bedford-McMullen carpets (see \cite{Fraser1}), and recently in \cite{K2}, Kesseböhmer and Niemann investigated the quantization dimension of negative order. Further investigation of the quantization dimension of invariant measures associated with conformal mappings, bi-Lipschitz mappings, can be explored in \cite{KPot, PVV1, Zhu1}.\\
Verma and Massopust \cite{VM1} introduced the approximation of any continuous function that preserves the dimensions (Hausdorff dimension and packing dimension) of the graphs. Later, some dense subsets of the set of all compact subsets of $\mathbb{R}^m$ w.r.t. the Hausdorff distance preserving the Assouad dimension are explored in \cite{VAM,VP1}. Mauldin and Williams \cite{Mauldin}, and later Liu and Zu \cite{Liu} commented on the decomposition of a continuous function that preserves the Hausdorff dimension of the graphs. The analogy of the previous for the set-valued function is illustrated in \cite{AV2}. Some recent work on Hausdorff and other fractal dimensions and set-valued fractal functions with applications can be seen in \cite{AV1,AV2,GP1, KBV11,KBV12,RSP,VAM,VP1}. Megala and Prasad \cite{MP1,MP2,MP3} studied the spectrum of self-affine invariant measures. It is still open whether the decomposition of a set, a continuous function, and a measure preserving the lower dimension.

\subsection{Delineation} In this paper, we study the density properties of various subsets of $\Pi(X)$, the collection of all compact subsets of a complete metric space $X$. Let $\Omega(X)$ denote the collection of all Borel probability measures on $X$, where $X$ is a Euclidean space $\mathbb{R}^m$ of dimension $m$ with Euclidean metric. We show that for any $0 \le \beta \le m$, the set of Borel probability measures with lower dimension $m$ forms a dense subset of $\Omega(X)$. Simultaneously, we demonstrate that the set of Borel probability measures with quantization dimension $\beta$ is also dense in $\Omega(X)$. Furthermore, in Subsection~\ref{subsec1}, we prove that the lower dimension of the invariant measure associated with a product IFS is equal to the sum of the lower dimensions of the invariant measures corresponding to the individual IFSs involved in the product.
\\

\begin{remark}
Sometimes, the study of complex measures can be reduced to the analysis of simpler ones, such as Dirac measures or smooth measures, which are often more tractable analytically. Approximating a given measure by such simpler measures plays a crucial role, particularly when the goal is to preserve important geometric features, such as dimensional characteristics. This motivates us to approximate a given measure $\vartheta$ by convolving it with a finite combination of Dirac measures, while ensuring that its dimensional properties are retained.
\end{remark}
 
\section{Preliminaries} \label{Sec2}
In this section, we recall some definitions and basic results for our paper. For any $F\subseteq\mathbb{R}^m,$ we denote $|F|=\sup_{x,y\in F}\|x-y\|_2$ and $\text{Card}(F)$ as the diameter and cardinality of $F$, respectively. \\
Let $\Pi(\mathbb{R}^m)$ denote the collection of all compact nonempty subsets of $\mathbb{R}^m$ with the Euclidean norm $\|\cdot\|_2$. The Hausdorff distance $\mathfrak{H}$ between any $E,F\in\Pi(\mathbb{R}^m)$ is elaborated as 
$$\mathfrak{H}(E,F):=\inf\{\epsilon>0: E\subset F_\epsilon \text{ and } F\subset E_\epsilon\},$$ where $E_\epsilon$ denotes the $\epsilon$-neighbourhood of $E.$
\begin{definition}\cite{Fraser}
    Let $E$ be a nonempty subset of $\mathbb{R}^m.$ The lower dimension of $E$ is defined as
    \begin{align*}
       \dim_{L}(E):=\sup&\Big\{\beta: \exists ~M>0 \text{ such that,} \forall ~0<r<R\le|E| \text{ and }x\in E, \\& \hspace{4cm}N_r(\mathcal{B}_R(x)\cap E)\ge M\Big(\frac{R}{r}\Big)^\beta\Big\},
    \end{align*}
where $\mathcal{B}_R(x)=\{y\in\mathbb{R}^m:\|y-x\|_2\le R\}$ and $N_r(\mathcal{B}_R(x)\cap E)$ is the smallest number of open set required for a $r$-cover of $\mathcal{B}_R(x)\cap E.$    
\end{definition}
The properties of lower dimension can be visited in the book \cite{Fraser}. It is noted that the lower dimension only satisfies stability under closure and bi-Lipschitz mapping.
\begin{definition}\cite{Fraser}
 Let $\vartheta$ be a Borel probability measure on $\mathbb{R}^m$ and $\text{supp}(\vartheta)$ denote the support of $\vartheta.$ Then the lower dimension of $\vartheta$ is defined by
 \begin{align*}
\dim_L(\vartheta)=\sup\Bigg\{&\beta\ge0:\exists ~M>0 \text{ such that,} \forall ~0<r<R\le|E| \text{ and }x\in \text{supp}(\vartheta),\\& \hspace{5.8cm}\frac{\vartheta(\mathcal{B}_R(x))}{\vartheta(\mathcal{B}_r(x))}\ge M\Big(\frac{R}{r}\Big)^\beta\Bigg\},
 \end{align*}
otherwise it is $0.$
    \end{definition}
We would like to mention here that if a measure $\vartheta$ is doubling and fully supported on a closed set $F,$ then $\dim_L(\vartheta)\le \dim_L(F),$
for details, see \cite[Lemma 4.1.2]{Fraser}.
\begin{definition}\cite{GL1}
    Let $\vartheta$ be a Borel probability measure on $\mathbb{R}^m,~r \in (0, +\infty)$ and $n \in \mathbb{N}.$  Then, the $n$th quantization error of order $r$ of $\vartheta$ is stated as:
    \[V_{n, r}(\vartheta):=\text{inf} \Big\{\int d(x, A)^r d\vartheta(x): A \subset \mathbb{R}^m, \ \text{Card}(A) \leq n\Big\},\]
    where $d(x,A)$ denotes the distance of point $x \in \mathbb{R}^m$ to the set $A$ with respect to the Euclidean norm $\|\cdot\|_2$ on $\mathbb{R}^m.$
    The quantization dimension of order $r$ of $\vartheta$ is defined by:
    $$ D_r(\vartheta):= \lim_{n\to \infty} \frac{r \log n}{- \log(V_{n,r}(\vartheta))},$$
    if the limit exists, otherwise, we define the lower and upper quantization dimensions of order $r$ of $\vartheta$ by taking the limit inferior and limit superior of the sequence, respectively.
\end{definition}
\begin{definition}\cite{GL1}
    For a given $s>0,$ the $s$-dimensional lower and upper quantization coefficient of order $r$ for $\vartheta$ is defined by:
    $$\liminf_{n \to \infty} n^{\frac{r}{s}}V_{n,r}(\vartheta),~ \text{and}~ \limsup_{n \to \infty}  n^{\frac{r}{s}}V_{n,r}(\vartheta),$$
    respectively.
\end{definition}
Let $\Omega(\mathbb{R}^m)$ denote the set of all Borel probability measures on the Euclidean space $(\mathbb{R}^m,\|.\|_2)$. Then,
\begin{equation*}
\begin{aligned}
d_L (\vartheta, \theta) :=\sup_{h\in\te{Lip}_1(\mathbb{R}^m)} \Big\{\Big|\int_{\mathbb{R}^m} h d\vartheta -\int_{\mathbb{R}^m} hd\theta\Big|\Big\}, \ \ \quad \vartheta, \theta \in \Omega(\mathbb{R}^m) , \end{aligned}
\end{equation*}
defines a metric on $\Omega(\mathbb{R}^m)$, where $\te{Lip}_1(\mathbb{R}^m)=\{h:\mathbb{R}^m\rightarrow \mathbb{R}: h\te{ is Lipschitz function}\\ \te{with Lipschitz constant }\le1\}$ (see \cite{Mat,Parth}). If $(X,d)$ is a compact metric space, then
$(\Omega(X), d_L)$ is a compact metric space (see \cite[Theorem~5.1]{B}).
\par



\begin{definition}
Let $\vartheta, \theta \in \Omega(\mathbb{R}^m).$ The convolution of $\vartheta$ and $\theta$ is denoted by $\vartheta *\theta$ and defined by the push-forward of the product measure $\vartheta \times \theta$ under the mapping $(x,y)\to x+y.$ That is, for any $E \subset \mathbb{R}^m$, we have $$\vartheta * \theta (E)= (\vartheta \times \theta )\big(\{(x,y) \in \mathbb{R}^m\times \mathbb{R}^m: x+y \in E\}\big).$$
\end{definition}
\begin{definition}
Let $\vartheta \in \Omega(\mathbb{R}^m)$ and $x \in \mathbb{R}^m.$ The translation of $\vartheta$ by $x$ is defined as 
\[
(\vartheta+x)(E)=\vartheta(E+x),~~~E \subset \mathbb{R}^m.
\]
\end{definition}

Let us define
\[
\mathcal{L}(\mathbb{R}^m)=\Big\{\vartheta 
\in \Omega(\mathbb{R}^m): \vartheta= \sum_{i=1}^k a_i \delta_{x_i},~x_i \in \mathbb{R}^m, ~a_i > 0, \sum_{i=1}^k a_i=1,~~ k \in \mathbb{N} \Big\}.
\]
Here $\delta_x$ denotes the Dirac measure supported at $x \in \mathbb{R}^m.$ It is observed that $\mathcal{L}(\mathbb{R}^m)$ is dense in $\Omega(\mathbb{R}^m)$ w.r.t. $d_L$ on $\Omega(\mathbb{R}^m),$ see \cite{Bill}. The existence of an invariant measure supported on the attractor of the associated iterated function system is shown in \cite{H}. Feng et al. \cite{Feng} studied the convolutions of equicontractive self-similar measures on $\mathbb{R}.$ 
We define essential supremum norm $\|.\|_{\mathcal{L}^{\infty}}$ of a measurable function $g$ with respect to a measure $\vartheta$ by 
\[
 \|g\|_{\mathcal{L}^{\infty}}= \inf \{C\ge 0 : |f(x)| \le C ~\text{for $\vartheta$-almost every x} \}.
\]
Let $\vartheta ,\lambda \in \Omega( \mathbb{R}^m)$ be such that $\vartheta$ is absolutely continuous with respect to $\lambda.$ Then by Radon-Nikodym theorem, there exists a measurable function $g: \mathbb{R}^m \to [0, \infty) $ such that $$ \vartheta(A)= \int_A g d\lambda,~~~\text{for any Borel set}~A \subset \mathbb{R}^m.$$
Further, the function $g$ is called the Radon-Nikodym derivative or density of $\vartheta$ with respect to $\lambda$, and is denoted by $\frac{ d\vartheta}{ d\lambda}.$
We call $\vartheta$ has $\mathcal{L}^{\infty}$-density if it is absolutely continuous with respect to $\lambda$, and density function $g$ satisfying
 $\|g\|_{\mathcal{L}^{\infty}} < \infty.$  
\par
Now we are ready to prove the upcoming result. Here and throughout the paper, we assume that $D_r(\vartheta)$ exists whenever it occurs.

\section{Main Results}
In the subsequent theorems, for the reader's better understanding, as well as for simplicity and clarity, we present the proof on $\mathbb{R}$, where the main ideas are more transparent. The results can be extended to $\mathbb{R}^m$ with suitable modifications, following a similar technique.
\begin{theorem} \label{thm3.1}
    For each $\gamma\in[0,1],$ the set $L^\gamma=\{A\in \Pi(\mathbb
    R): \dim_L(A)=\gamma\}$ is dense in $\Pi(\mathbb
    R).$ 
\end{theorem}
\begin{proof}
Let $A \subseteq \mathbb{R}$ be a compact set. Then, for any $\epsilon > 0$, there exist points $a_1, a_2, \ldots, a_n \in A$ and open balls $\mathcal{B}_\epsilon(a_1), \mathcal{B}_\epsilon(a_2), \ldots, \mathcal{B}_\epsilon(a_n)$ with radius $\epsilon$ such that $A \subset \bigcup_{j=1}^n \mathcal{B}_\epsilon(a_j)$, and $\mathcal{B}_\frac{\epsilon}{2}(a_i) \cap \mathcal{B}_\frac{\epsilon}{2}(a_j) = \emptyset$ for $i \ne j$. Define the finite set $A_* = \{a_1, a_2, \ldots, a_n\} \in \Pi(\mathbb{R})$, without any loss one may assume $a_1<a_2< \dots < a_n$. Since the lower dimension of a set with isolated points is zero (see \cite{Fraser}), then $\dim_L(A_*) = 0$. In view of the definition of the Hausdorff metric, we obtain  
\[
\mathfrak{H}(A, A_*) < \epsilon \quad \text{and} \quad \dim_L (A_*) = 0.
\]  
This establishes the claim for the case $ \gamma=0$.  
Now, assume $\gamma \in (0, 1]$. By \cite[Theorem 1]{Chen}, there exists a compact set $A_1 \subset \mathcal{B}_\frac{\epsilon}{2}(a_1)$ such that $\dim_L (A_1) = \gamma$. Since the lower dimension is translation invariant, we may assume without loss of generality that $a_1 \in A_1$.  
Choose $\delta = a_2 - a_1$, and consider the $\delta$-translation of $A_1$, denoted by $A_2$. Then $a_2 \in A_2$. Moreover, for any $y \in A_2$, there exists $x \in A_1$ such that  
\[
y = x + \delta = x + (a_2 - a_1).
\]  
It follows that  
\[
d(y, a_2) = |x + (a_2 - a_1) - a_2| = |x - a_1| < \tfrac{\epsilon}{2}.
\] Thus, $A_2 \subset \mathcal{B}_\frac{\epsilon}{2}(a_2)$ with $\dim_L (A_2) = \gamma$. Repeating this procedure for each $j$, we construct sets $A_j \subset \mathcal{B}_{\frac{\epsilon}{2}}(a_j)$ satisfying $\dim_L (A_j) = \gamma$ and $A_i \cap A_j =\emptyset$ for $i \ne j$.  
Define $A' = \bigcup_{j=1}^n A_j$. Since for $i \neq j$, $A_i$ and $A_j$ are disjoint, it follows that $\inf_{x \in A_i, y \in A_j} d(x,y)>0.$ Then, by \cite[Lemma 3.4.9]{Fraser}, we have $\dim_L (A') = \gamma,$ and by the definition of the Hausdorff metric and construction of the set $A$, we obtained $\mathfrak{H}(A, A') < \epsilon,$ which completes the proof.
\end{proof}

 \begin{remark}
The proof given in \cite[Theorem 3.8]{VAM} is not applicable in the case of the lower dimension, since, unlike the Assouad dimension, the lower dimension does not satisfy the open set and the finite stability property.
 \end{remark}
It is noted that in general lower dimension is not comparable to the Hausdorff dimension of set, as an example $\dim_L(\mathbb{Q})=1, ~\dim_\mathcal{H}(\mathbb{Q})=0$ and $\dim_L([0,1]\cup\{2\})=0, ~\dim_\mathcal{H}([0,1]\cup\{2\})=1$ (see \cite{Fraser}). Larman \cite{Larman} showed that if $F$ is a compact subset of $\mathbb{R}^m,$ then 
$\dim_L(F)\le \dim_{\mathcal{H}}(F).$ Further, if $F$ is compact and Ahlfors regular, then $\dim_L(F)=\dim_\mathcal{H}(F)$ (see \cite[Theorem 6.4.1]{Fraser}). Thus, it is worth mentioning next few results on the approximation of a compact set preserving the lower dimension and the Hausdorff dimension.  
 \begin{theorem}
    The set $\mathcal{S}_{\mathcal{H}L}=\{A\in\Pi(\mathbb{R}):\dim_\mathcal{H}(A)\ne\dim_L(A)\}$ is dense in $\Pi(\mathbb{R}).$ 
\end{theorem}
\begin{proof}
  Observe that the set $\mathcal{S}_{\mathcal{H}L}$ is non-empty, as the Bedford–McMullen carpet with non-uniform fibers has the Hausdorff dimension strictly greater than its lower dimension (see \cite[Theorem 2.1]{Fraser1}). Also, for any bounded closed interval $I$ and $a \in \mathbb{R}$ such that $a \notin I$, the Hausdorff dimension of $I \cup \{a\}$ is $1$, but the lower dimension is $0$. This concludes that the set $\mathcal{S}_{\mathcal{H}L}$ is non-empty.
Let $A \subset \mathbb{R}$ be a compact set. Then, for any $\epsilon > 0$, there exist points $a_1, a_2, \dots, a_n$ and open balls $\mathcal{B}_{\epsilon}(a_1), \mathcal{B}_{\epsilon}(a_2), \ldots, \mathcal{B}_{\epsilon}(a_n)$ such that $A \subset \bigcup_{j=1}^n \mathcal{B}_{\epsilon}(a_j)$, with $\mathcal{B}_{\frac{\epsilon}{2}}(a_i) \cap \mathcal{B}_{\frac{\epsilon}{2}}(a_j) = \emptyset$ for $i \neq j$. Define $A_* = \mathcal{B}_{\epsilon}(a_1) \cup \{a_2, \ldots, a_n\}$. Since the Hausdorff dimension satisfies the open set property and $A_*$ has isolated points, we have $\dim_\mathcal{H} (A_*) = 1$ and $\dim_L (A_*) = 0$. Furthermore, by the definition of $A_*$, we have $\mathfrak{H}(A, A_*) < \epsilon$. This completes the assertion.
 
\end{proof}
\begin{theorem}
    The set $\mathcal{D}=\{A\in\Pi(\mathbb{R}):\dim_\mathcal{H}(A) = \dim_L(A)\}$ is dense in $\Pi(\mathbb{R}).$ 
\end{theorem}
\begin{proof}
    Since the Hausdorff and lower dimensions of the singleton set $\{a\}$, for any $a \in \mathbb{R}$, are zero, the set $\mathcal{D}$ is non-empty. Let $A \in \Pi(\mathbb{R})$ be a compact set. Then, for any $\epsilon > 0$, there exist points $a_1, a_2, \ldots, a_n \in A$ such that $A = \bigcup_{j=1}^n \mathcal{B}_{\epsilon}(a_j)$,
  where $\mathcal{B}_{\epsilon}(a_j)$ denotes the open ball around $a_j$ of radius $\epsilon$. We may assume that the balls $\mathcal{B}_{\frac{\epsilon}{2}}(a_i)$ and $\mathcal{B}_{\frac{\epsilon}{2}}(a_j)$ are disjoint for $i \neq j$. Define  
\[
A' = \bigcup_{j=1}^n \mathcal{B}_{\tfrac{\epsilon}{2}}(a_j).
\]  
Since the open balls $\mathcal{B}_{\frac{\epsilon}{2}}(a_i)$ and $\mathcal{B}_{\frac{\epsilon}{2}}(a_j)$ are disjoint for $i \neq j$, it follows that  
\[
\inf_{x \in \mathcal{B}_{\tfrac{\epsilon}{2}}(a_i), \, y \in \mathcal{B}_{\tfrac{\epsilon}{2}}(a_j)} d(x, y) > 0 \quad \text{for } i \neq j.
\]  
Then, by \cite[Lemma 3.4.9]{Fraser}, we have $\dim_L(A') = 1$. Also, by the countable stability and the open set property of the Hausdorff dimension, we have $\dim_\mathcal{H} (A') = 1$. In view of the definition of the Hausdorff metric and the construction of the set $A'$, we obtain $\mathfrak{H}(A, A') < \epsilon$, which completes the proof.
\end{proof}

\begin{theorem}
    For each $\gamma \in [0,1].$ The set $\mathcal{D}=\{A\in\Pi(\mathbb{R}):\dim_\mathcal{H}(A) = \dim_L(A)= \gamma\}$ is dense in $\Pi(\mathbb{R}).$ 
\end{theorem}
\begin{proof}
In the case of \( \gamma = 0 \), for any \( \epsilon > 0 \), we construct the set \( A_* \) as in Theorem \ref{thm3.1}. Then \( \dim_L(A_*) = \dim_\mathcal{H}(A_*) = 0 \), and \( \mathfrak{H}(A, A_*) < \epsilon \). This shows that the assertion is true for \( \gamma = 0 \).
Now, let \( \gamma \in (0,1] \). Then there exists an Ahlfors-regular set \( A \) of dimension \( \gamma \) (see \cite{Arc}). Following the same approach as in the proof of Theorem \ref{thm3.1}, the result can be obtained. Hence, we omit the proof.

\end{proof}
\subsection{Applications in dimension-preserving approximation of measures}
This subsection is dedicated for the approximation of a probability measure preserving lower dimension and quantization dimension. Furthermore, it addresses whether, if one probability measure is absolutely continuous with respect to another probability measure, one can establish a relationship between their quantization dimension. We begin by establishing some lemmas necessary for our main results.

\begin{lemma}\label{convothm}
    Let $\mathcal{L}(\mathbb{R})=\big\{\sum_{j=1}^na_j\delta_{x_j}: x_j\in\mathbb{R},a_j>0,\sum_{j=1}^na_j=1,n\in\mathbb{N}\big\}$ and $\vartheta$ be a probability measure with finite support. Then, for a fixed $\omega\in\mathcal{L}(\mathbb{R}),$  $$\dim_{L}(\vartheta*\omega)=\dim_{L}(\vartheta),$$
    where $\vartheta*\omega$ is the convolution of $\vartheta$ and $\omega.$
\end{lemma}
\begin{proof}
  Let $\vartheta$ be a probability measure with finite support. We define  sets 
   \begin{align*}
     A_{\vartheta} = \Bigg\{&\beta\ge0: \exists~ M>0 \text{ such that}, \forall~ 0<r<R<|\text{supp} (\vartheta)|\text{ and }x\in\text{supp}(\vartheta), \\&\hspace{7cm}\frac{\vartheta(\mathcal{B}_R(x))}{\vartheta(\mathcal{B}_r(x))}\ge M\Big(\frac{R}{r}\Big)^\beta\Bigg\},
 \end{align*}
  \begin{align*}
     A_{\vartheta*\omega} = \Bigg\{&\gamma\ge0:\exists~ M>0 \text{ such that}, \forall~ 0<r<R<|\text{supp} (\vartheta * \omega)|\text{ and } \\&\hspace{3.6cm}x\in\text{supp}(\vartheta *\omega), \frac{\vartheta*\omega(\mathcal{B}_R(x))}{\vartheta*\omega(\mathcal{B}_r(x))}\ge M\Big(\frac{R}{r}\Big)^\gamma\Bigg\}.
 \end{align*}
  Then, by the definition of the lower dimension for any $s \in A_{\vartheta}$, there exists a constant $M$ for any $x \in \text{supp}(\vartheta)$ and $0< r<R< |\text{supp}(\vartheta)|,$ we have
  \[
  \frac{\vartheta(\mathcal{B}_R(x))}{\vartheta(\mathcal{B}_r(x))} \ge M\bigg(\frac{R}{r}\bigg)^s.
  \] 
  Note that, for the Dirac measure $\delta_{y}$ supported on $y \in \mathbb{R}$ the support of the convolution $\text{supp}(\vartheta*\delta_{y})= \text{supp}(\vartheta) +\{y\}$, for each $x \in \text{supp}(\vartheta*\delta_{y})$, we have $z=x-y \in \text{supp}(\vartheta)$.
  Now, let $\omega = \sum_{j=1}^n a_j \delta_{x_j}$ be a probability measure in $\mathcal{L}(\mathbb{R})$; then for any $x \in \text{supp}(\vartheta*\omega)$ and $0<r < R< |\text{supp}(\vartheta*\omega)|,$ we have 
\begin{equation*}
    \begin{aligned}
      \frac{(\vartheta * \omega)(\mathcal{B}_R(x))}{(\vartheta * \omega)(\mathcal{B}_r(x))} &= \frac{(\vartheta * \sum_{j=1}^n a_j \delta_{x_j}) (\mathcal{B}_R(x))}{(\vartheta * \sum_{j=1}^n a_j \delta_{x_j}) (\mathcal{B}_r(x))} = \frac{\sum_{j=1}^n a_j \vartheta( \mathcal{B}_R(x)-x_j)}{\sum_{j=1}^n a_j \vartheta( \mathcal{B}_r(x)-x_j)} \\& = \frac{\sum_{j=1}^n a_j \vartheta( \mathcal{B}_R(x-x_j))}{\sum_{j=1}^n a_j \vartheta( \mathcal{B}_r(x- x_j,))}   \ge \frac{\sum_{j=1}^n a_j M \bigg(\frac{R}{r}\bigg)^s \vartheta(\mathcal{B}_r(x-x_j))}{\sum_{j=1}^n a_j \vartheta( \mathcal{B}_r(x- x_j))} \\& \ge M \bigg(\frac{R}{r}\bigg)^s.
    \end{aligned}
\end{equation*}
In the above expression, the fourth inequality arrived due to the fact that $ x-x_j \in \text{supp}(\vartheta).$ Thus, it implies that $s \in A_{\vartheta * \omega}.$
Again, let $s \in A_{\vartheta}^C$ be any real number. Then, for some $x \in \text{supp}(\vartheta)$ and $0< r< R< |\text{supp}(\vartheta)|$, we have 
$$  \frac{\vartheta(\mathcal{B}_R(x))}{\vartheta(\mathcal{B}_r(x))} < M\bigg(\frac{R}{r}\bigg)^s.$$
Notice that for $x \in \text{supp}(\vartheta),~ x +x_j \in \text{supp}(\vartheta * \omega).$ Consider 
\begin{equation*}
   \begin{aligned}
        \frac{(\vartheta * \omega)(\mathcal{B}_R(x+x_j))}{(\vartheta * \omega)(\mathcal{B}_r(x+x_j))} = \frac{\sum_{j=1}^n a_j \vartheta(\mathcal{B}_R(x))}{\sum_{j=1}^n a_j \vartheta( \mathcal{B}_r(x))} &< \frac{\sum_{j=1}^n a_j M \bigg(\frac{R}{r}\bigg)^s \vartheta(\mathcal{B}_r(x))}{\sum_{j=1}^n a_j \vartheta( \mathcal{B}_r(x))} \\& < M \bigg(\frac{R}{r}\bigg)^s.
   \end{aligned} 
\end{equation*}
This implies that $s \in A_{\vartheta *\omega}^C,$ concluding the assertion.
\end{proof}

\begin{lemma}\label{1179}
    Let $\vartheta$ be a Borel probability measure on $\mathbb{R}^m$ and $\beta$ be a positive real number. Then $\dim_{L}(\vartheta)=\dim_{L}(\theta),$ where $\theta(E)=\vartheta(\beta E),$ for all Borel set $E.$  
\end{lemma}
\begin{proof}
    Let $A_{\vartheta}$ and $A_{\theta}$ be the sets defined similarly as in Lemma \ref{convothm} with respect to the measures $\vartheta$ and $\theta$, respectively. Observe that support of $\theta$ is $\frac{1}{\beta} \text{supp}(\vartheta),$ and $|\text{supp}(\theta)|= |\frac{1}{\beta} \text{supp}(\vartheta)|.$ For any $s \in A_{\vartheta}, ~~x \in \text{supp}(\theta),$ and $0< r < R < |\text{supp}(\theta)|,$ we have
\[
\frac{\theta(\mathcal{B}_R(x))}{\theta(\mathcal{B}_r(x))} = \frac{\theta(\mathcal{B}_R(\frac{y}{\beta}))}{\theta(\mathcal{B}_r(\frac{y}{\beta}))},
\]
for some $y \in \text{supp}(\vartheta).$ Also, $\beta \mathcal{B}_r(x)= \mathcal{B}_{\beta r}(\beta x),$ and $0< \beta r < \beta R <|\text{supp}(\vartheta)|$ whenever $0 < r < R< |\text{supp}(\theta)|.$ Now, since $\theta(D)= \vartheta(\beta D)$ for any Borel set $D$, we obtain
\begin{equation*}
    \begin{aligned}
     \hspace{1.9cm} \frac{\theta(\mathcal{B}_R(\frac{y}{\beta}))}{\theta(\mathcal{B}_r(\frac{y}{\beta}))} = \frac{\vartheta(\beta \mathcal{B}_R(\frac{y}{\beta}))}{\vartheta(\beta \mathcal{B}_r(\frac{y}{\beta}))} = \frac{\vartheta(\mathcal{B}_{\beta R}(y))}{\vartheta(\mathcal{B}_{\beta r}(y))}  \ge M \bigg(\frac{R}{r}\bigg)^s.  
    \end{aligned}
\end{equation*}
This implies that $s \in A_{\theta}.$
Similarly for any $s \in A_{\vartheta}^C,$ we obtained
$\frac{\theta(\mathcal{B}_R(x))}{\theta(\mathcal{B}_r(x))} = \frac{\vartheta(\mathcal{B}_{\beta R}(y))}{\vartheta(\mathcal{B}_{\beta r}(y))} < M \bigg(\frac{R}{r}\bigg)^s,$ for some $x \in \text{supp}(\vartheta)$ and $0 < r < R < |\text{supp}(\vartheta)|.$ this implies that $s \in A_{\theta}^C.$
Thus, we get the required result.
\end{proof}
 \begin{lemma}\label{lipdim}
Let $\vartheta, \theta$ be finite Borel measures. Then we have that 
\[
D_r(\vartheta+\theta) = \max\{D_r(\theta),D_r(\vartheta)\}.
\]
 \end{lemma}
 \begin{proof}
 We have
 \begin{equation*}
 \begin{aligned}
 V_{n,r}(\vartheta+\theta)= & \inf \Big\{\int d(x, A)^r d(\vartheta+\theta)(x): A \subset \mathbb{R}^m,~ \text{Card}(A) \le n\Big\}\\ \ge &\inf \Big\{\int d(x, A)^r d\vartheta(x): A \subset \mathbb{R}^m, ~ \text{Card}(A) \le n\Big\} \\ & + \inf \Big\{\int d(x, A)^r d\theta(x): A \subset \mathbb{R}^m, ~ \text{Card}(A) \le n\Big\} \\  = & V_{n,r}(\vartheta)+V_{n,r}(\theta).
 \end{aligned}
 \end{equation*}
 By the definition of quantization dimension, it follows that $$D_r(\vartheta+\theta) \ge \max\{D_r(\vartheta),D_r(\theta)\}.$$
 Now, let $s>  \max\{D_r(\vartheta),D_r(\theta)\}.$ In view of \cite[Propostion $11.3$]{GL1}, we have $$ \lim_{n \to \infty} n^{\frac{r}{s}} V_{n,r}(\vartheta)=0 ~~~\text{and}~~~~  \lim_{n \to \infty} n^{\frac{r}{s}} V_{n,r}(\theta)=0.$$
 Let $\ep > 0.$ Then there exist natural numbers $N_1(\vartheta,\ep)$ and $N_2(\theta,\ep)$ such that $$ n^{\frac{r}{s}} V_{n,r}(\vartheta) < \ep ~~~ \forall~~ n \ge N_1, ~\text{and}~~ n^{\frac{r}{s}} V_{n,r}(\theta) < \ep ~~~ \forall~~ n \ge N_2.$$
 Choose $N_0=\max\{N_1,N_2\},$ we immediately have 
 $$ n^{\frac{r}{s}} V_{n,r}(\vartheta) < \ep  ~\text{and}~~ n^{\frac{r}{s}}  V_{n,r}(\theta) < \ep ~~~ \forall~~ n \ge N_0.$$ 
 Further, using the definition of $n$-th quantization error of order $r$ of a measure, there exist sets $A(n,\vartheta,\ep)$ and $B(n,\theta,\ep)$ satisfying the following:
 $$ n^{\frac{r}{s}}  \int d\Big(x, A(n,\vartheta,\ep)\Big)^r d\vartheta(x) < \ep ~~\text{ and }~~   n^{\frac{r}{s}} \int d\Big(x, B(n,\theta,\ep)\Big)^r d\theta(x) < \ep ~~ \forall ~~n \ge N_0 .$$
 Since $\text{Card}\Big(A(n,\vartheta,\ep)\Big) \le n $ and $\text{Card}\Big(B(n,\theta,\ep)\Big) \le n $, it follows that 
 \begin{equation*}
 \begin{aligned}
\hspace{1cm}&\text{Card}\Big(A(n,\vartheta,\ep) \cup B(n,\theta,\ep)\Big) \le 2n,\quad\text{and} \\& 
(2n)^{\frac{r}{s}} \int d\Big(x, A(n,\vartheta,\ep) \cup B(n,\theta,\ep)\Big)^r d\vartheta(x) < 2^{\frac{r}{s}} \ep,
 \end{aligned}
 \end{equation*}
 for every $n \ge N_0.$ 
 Using all the above estimates, we have
 \begin{equation*}
  \begin{aligned}
  (2n)^{\frac{r}{s}}& V_{2n,r}(\vartheta+\theta)\\= & (2n)^{\frac{r}{s}}~ \inf \Big\{\int d(x, A)^r d(\vartheta+\theta)(x): A \subset \mathbb{R}^m,~ \text{Card}(A) \le 2n\Big\}\\  = & (2n)^{\frac{r}{s}} ~\inf \Big\{\int d(x, A)^r d\vartheta(x)+\int d(x, A)^r d\theta(x): A \subset \mathbb{R}^m,~ \text{Card}(A) \le 2n\Big\}\\  \le  & (2n)^{\frac{r}{s}} \int d\Big(x, A(n,\vartheta,\ep) \cup B(n,\theta,\ep)\Big)^r d\vartheta(x)+ \\& \hspace{1cm}(2n)^{\frac{r}{s}}\int d\Big(x, A(n,\vartheta,\ep) \cup B(n,\theta,\ep)\Big)^r d\theta(x)\\ < & ~2^{\frac{r}{s}+1} \ep.
  \end{aligned}
  \end{equation*}
  Since $\lim_{n \to \infty} (2n)^{\frac{r}{s}} V_{2n,r}(\vartheta+\theta)=0,$ we have $D_r(\vartheta+\theta) < s.$ Further, since $s > \max\{D_r(\vartheta),D_r(\theta)\}$ is arbitrary, it follows that $D_r(\vartheta+\theta) \le \max\{D_r(\vartheta),D_r(\theta)\}.$ This completes the proof of the lemma.
 \end{proof}
 \begin{remark}
 The result of the above lemma can be compared to those that appeared in \cite{Mat3} in terms of Hausdorff dimension. For example, let us take $\lambda= \vartheta + \theta,$ where $\vartheta=\mathcal{L}^1|_{[0,1]}$ and $ \theta \in \mathcal{L}(\mathbb{R}).$ Then it can be deduced that $\dim_\mathcal{H}(\lambda)=\min\{\dim_\mathcal{H}(\vartheta),\dim_\mathcal{H}(\theta)\}=0,$ however, $D_r(\lambda)=\max\{D_r(\vartheta), D_r(\theta)\}=1.$ 
 \end{remark}
\begin{lemma}\label{lipdim1}
 Let $\vartheta$ be a finite Borel measure. For a fixed $\theta \in \mathcal{L}(\mathbb{R}^m)$, we have that 
 \[
 D_r(\vartheta*\theta) =D_r(\vartheta).
 \]
  \end{lemma}
  \begin{proof}
   Let $\theta \in \mathcal{L}(\mathbb{R}^m).$ By definition of $\mathcal{L}(\mathbb{R}^m),$ we write $\theta= \sum_{i=1}^k a_i \delta_{x_i}$ for some $x_i \in \mathbb{R}^m$ and $a_i \ge 0$ such that $\sum_{i=1}^k a_i=1.$ Now, note that
    \[
    \vartheta*\theta=\vartheta *\Big(\sum_{i=1}^k a_i \delta_{x_i}\Big)= \sum_{i=1}^k a_i \vartheta*\delta_{x_i}.
    \]
    Using Lemma \ref{lipdim}, we have 
    \[
    D_r(\vartheta*\theta)=D_r\Big(\sum_{i=1}^k a_i \vartheta*\delta_{x_i}\Big)= \max_{1\le i\le k}\{D_r(a_i \vartheta*\delta_{x_i})\}=\max_{1\le i\le k}\{D_r(\vartheta*\delta_{x_i})\}.
    \]
    From the above, it is enough to prove that $D_r(\vartheta*\delta_{x_i})= D_r(\vartheta).$ We proceed as follows.
    \begin{equation*}
    \begin{aligned}
    V_{n,r}(\vartheta*\delta_{x_i}) & =  \inf \Big\{\int d(x, A)^r d(\vartheta*\delta_{x_i})(x): A \subset \mathbb{R}^m, \, \text{Card}(A) \le n\Big\}\\
    & =  \inf \Big\{\int d(x+x_i, A)^r d\vartheta(x): A \subset \mathbb{R}^m, \, \text{Card}(A) \le n\Big\}\\& =  \inf \Big\{\int d(x+x_i, B+x_i)^r d\vartheta(x): B=A-x_i \subset \mathbb{R}^m, \, \text{Card}(A) \le n\Big\}\\
    & =  \inf \Big\{\int d(x, B)^r d\vartheta(x): B \subset \mathbb{R}^m, \, \text{Card}(B) \le n\Big\}\\ & = V_{n,r}(\vartheta).
    \end{aligned}
    \end{equation*}
     This completes the proof of the lemma.
  \end{proof}
  \begin{remark}
    A related result can be found in Zhu's thesis \cite{Zhu2}, where he considered measures \(\vartheta\) and \(\theta\) satisfying the \(r + \delta\)-moment condition and established a relation between their upper quantization dimensions. Specifically, he proved that
$\overline{\dim}_r(\vartheta * \theta) \le \overline{\dim}_r(\vartheta) + \overline{\dim}_r(\theta).$ In contrast, our result shows that the quantization dimension of a measure \(\vartheta\) remains invariant under convolution with a finite combination of Dirac measures. Moreover, the approach we employ differs significantly from that used by Zhu.
  \end{remark}
  \begin{note}
      It is worth noting that for any compactly supported Borel probability measure $\vartheta,$ the $pth$ moment is always finite for all $p>0.$ Let $\vartheta$ be a Borel probability measure with compact support, then for any $x \in \text{supp}(\vartheta)$ there exists some $K >0$ such that $\|x\| \le K,$ and hence $\int \|x\|^p d\vartheta(x) \le K^p \int d \vartheta(x)= K^p < \infty.$  The finite $pth$ moment guaranteed the existence of lower and upper quantization dimension, but without this assumption, the quantization dimension can cross the dimension of the underlying space. In this way, our assumption is more general, and by following the same technique, one can obtain similar results for lower and upper quantization dimensions, respectively.
  \end{note}
  It is well-known \cite{Bill} that $\mathcal{L}(\mathbb{R}^m)$ is a dense subset of $\Omega(\mathbb{R}^m),$ where the latter is endowed with the metric $d_L.$ We use this fact to prove the following theorem.
\begin{theorem}\label{densethm}
Let $ \alpha \in [0,m]$. Then, the set  $\mathcal{S}_{\alpha}:=\{ \vartheta \in  \Omega(\mathbb{R}^m): D_r(\vartheta)  = \alpha\}$ is a dense subset of $\Omega(\mathbb{R}^m).$
\end{theorem}
\begin{proof}
    Let $ \vartheta \in \Omega(\mathbb{R}^m)$ and $\epsilon>0.$ Using the density of $\mathcal{L}(\mathbb{R}^m)$ in $\Omega(\mathbb{R}^m)$, there exists $\theta$ in $\mathcal{L}(\mathbb{R}^m)$ such that $$ d_L(\vartheta, \theta) < \frac{\epsilon}{2}.$$ Further, we consider a non-zero compactly supported measure $ \lambda \in \mathcal{S}_{\alpha}$ with its support $K \subset \mathbb{R}^m.$ Without loss of generality, we assume $K \subseteq [0,c]^m$ for some $c>0.$ Let $\lambda'= \theta * \lambda_1,$ where $\lambda_1(A):=\lambda(\frac{2|K| A}{\epsilon})$ and $|K|:=\sup\{\|x-y\|_2: x,y \in K\}.$ Note that Lemma \ref{lipdim1} implies that $D_r(\lambda')=D_r(\lambda)=\alpha,$ that is, $\lambda' \in \mathcal{S}_{\alpha}.$ On similar lines of \cite[Theorem 3.19]{VAM}, we get
      $$ d_L(\vartheta,\lambda') \le  d_L(\vartheta,\theta) + d_L(\theta,\lambda') < \epsilon .$$ 
             Thus, the proof of the theorem is complete.
\end{proof}

\begin{theorem}\label{densethm112}
    For each $\beta \in [0,m]$, the set $\mathcal{S}_\beta:=\{\vartheta\in\Omega(\mathbb{R}^m): \dim_{L}(\vartheta)= \beta\}$ is dense in $\Omega(\mathbb{R}^m)$.
\end{theorem}
\begin{proof}
 Observe first that \(\mathcal{S}_{\beta}\) is non‑empty for every \(\beta\in[0,m]\). Indeed, let \(A\subset\mathbb{R}\) be compact with \(m\ge \dim_L(A)>0\). Then for each \(\alpha\in[0,\dim_L(A)]\), there exists a probability measure \(\vartheta\) supported on \(A\) with \(\dim_L(\vartheta)=\alpha\) (see \cite[Theorem 4.1]{Hare}). Moreover, this result holds in any complete metric space \cite{Suom}. Consequently, \(\mathcal{S}_{\beta}\neq\emptyset\) for all \(\beta\in[0,m]\). Let $\vartheta \in \Omega({\mathbb{R}^m})$ and $\epsilon>0$, then by the density of $\mathcal{L}({\mathbb{R}}^m)$ in $\Omega({\mathbb{R}^m}),$ there exists a measure $\theta \in \mathcal{L}({\mathbb{R}}^m)$ such that $d_L(\vartheta,\theta)< \epsilon.$ Note that any Dirac measure or finite combination thereof has lower dimension zero, so $\dim_L(\theta)=0.$. This implies that the assertion is true for $\beta=0.$ Now assume that $\beta \in (0,m]$, then for any $\epsilon >0$ there exists $\theta \in \mathcal{L}(\mathbb{R}^m)$ such that $d_L(\vartheta,\theta)< \frac{\epsilon}{2}.$ Further, for each $\beta$ there exists $\lambda \in \mathcal{S}_{\beta}$ with compact support $K \subset \mathbb{R}^m,$ and $\dim_L(\lambda)= \beta.$ Define $\lambda_1 (A) := \lambda \big( \frac{2|K| A}{\epsilon}\big),$ then by Lemma \ref{1179}, $\dim_L(\lambda_1)= \dim_L(\lambda),$ and thus $\lambda_1 \in \mathcal{S}_{\beta}.$ Next, define the convolution $\lambda^*= \theta * \lambda_1.$ By Lemma \ref{convothm}, we obtain $\dim_L(\lambda^*)= \dim_L(\lambda_1)= \dim_L(\lambda)= \beta.$ This implies that $\lambda^* \in \mathcal{S}_{\beta}.$ On similar lines of \cite[Theorem 3.19]{VAM}, we get $d_L(\theta, \lambda^*)\le \frac{\epsilon}{2}.$ Therefore,
 $$d_L(\vartheta,\lambda^*) \le d_L(\vartheta,\theta)+ d_L(\theta, \lambda^*) < \epsilon.$$ This completes the proof of the theorem.
\end{proof}


\begin{remark}\label{rem24}
   It is known that for any closed set \( E \subset \mathbb{R}^m \), the lower dimension satisfies $\dim_L(E) = \sup\{ \dim_L(\vartheta) : \operatorname{supp}(\vartheta) = E \}$ (see \cite[Theorem 4.1.3]{Fraser}). In this context, it is worth emphasizing that if \( E \) is a closed set containing isolated points, then any Borel probability measure \( \vartheta \) supported on \( E \) satisfies \( \dim_L(\vartheta) = 0 \). In particular, this implies that the lower dimension of any discrete Borel probability measure is zero; that is, \( \dim_L(\vartheta) = 0 \) for any such \( \vartheta \).
 \end{remark}

    
In the following theorem, we prove the relationship between the quantization dimensions of two measures when one is absolutely continuous with respect to the other. A similar result was proved by Zhu in his thesis \cite{Zhu2}; however, our approach is different.
\begin{theorem}
Let $\vartheta, \theta \in \Omega(\mathbb{R}^m)$ such that $\vartheta$ is absolutely continuous with respect to $\theta$ with $\mathcal{L}^{\infty}$-density. Then, $ D_r(\vartheta) \le D_r(\theta).$
\end{theorem}
\begin{proof}
By Radon-Nikodym theorem and hypothesis taken in the theorem, there exists a measurable $\mathcal{L}^{\infty}$-function $g: \mathbb{R}^m \to [0, \infty) $ such that $$ \vartheta(A)= \int_A g d\theta,~~~\text{for any Borel set}~A \subset \mathbb{R}^m.$$
Now, we have
\begin{equation*}
 \begin{aligned}
 V_{n,r}(\vartheta)= & \inf \Big\{\int d(x, A)^r d\vartheta(x): A \subset \mathbb{R}^m, \, \text{Card}(A) \le n\Big\}\\ = &\inf \Big\{\int d(x, A)^r g(x) d\theta(x): A \subset \mathbb{R}^m, \, \text{Card}(A) \le n\Big\} \\ \le & ~\|g\|_{\mathcal{L}^{\infty}} \inf \Big\{\int d(x, A)^rd\theta(x) : A \subset \mathbb{R}^m, ~ \text{Card}(A) \le n\Big\} \\   = & ~ \|g\|_{\mathcal{L}^{\infty}} V_{n,r}(\theta),
 \end{aligned}
 \end{equation*}
 where $\|g\|_{\mathcal{L}^{\infty}}$ represents essential supremum norm of $g$ with respect to measure $\theta.$ 
 This further yields
  $$D_r(\vartheta)= \lim_{n \to \infty} \frac{r \log n}{- \log \big(V_{n,r}(\vartheta)\big)} \le \lim_{n \to \infty} \frac{r \log n}{- \log \big(\|g\|_{\mathcal{L}^{\infty}}V_{n,r}(\theta)\big)} = D_r(\theta).$$
  Thus, the proof of the theorem is established.
\end{proof}

\begin{remark}\label{Remuse1}
Let $\vartheta, \theta \in \Omega(\mathbb{R}^m).$ Denoting $\vartheta \le C \theta$ for some $C>0$, if $\vartheta(A) \le C \theta(A)$ holds for every Borel set $A \subset \mathbb{R}^m.$
Note that $\vartheta \le C \theta $ if and only if $\vartheta << \theta$ with an $\mathcal{L}_{\infty}$-density. Further, it is immediate that $$ D_r(\vartheta) \le D_r(\theta)~~ ~~~~\text{whenever}~ \vartheta \le C \theta.$$
\end{remark}

In the following examples, we demonstrate that there are measures for which the lower dimension, the Hausdorff dimension, and the quantization dimension exist and are equal. We also give an example of a probability measure for which these three dimensions are different.
\begin{example}
  Let $\vartheta = \sum_{i=1}^N a_i \delta_{x_i},$ where $a_i >0, \sum_{i=1}^N a_i = 1.$ Then we can observe that if we choose the set $A = \{x_1, x_2, \ldots, x_N\}$ in the definition of the quantization error, then for any $r>0$ and $n \ge N,$ the quantization error $nth$ of order $r$ is zero. In addition, its quantization dimension is zero. Also, by Remark \ref{rem24} and inequality $\dim_L(\vartheta) \le \dim_{\mathcal{H}}(\vartheta) \le D_r(\vartheta),$ we have $\dim_L(\vartheta) = \dim_{\mathcal{H}}(\vartheta) = D_r(\vartheta)=0.$
\end{example}
\begin{remark}
 In this way, for every measure belonging to $\mathcal{L}(\mathbb{R}^m)$, these three dimensions are equal to zero.   
\end{remark}
\begin{example}
Let \( \mathcal{I} = \{ f_1(x) = \frac{x}{3},\ f_2(x) = \frac{x}{3} + \frac{2}{3} \} \) be a self-similar iterated function system (IFS) consisting of two similarity mappings \( f_1 \) and \( f_2 \), each with the similarity ratio \( c_1 = c_2 = \frac{1}{3} \). Consider the probability vector \( \{ p_1 = \frac{1}{3},\ p_2 = \frac{2}{3} \} \). Graf and Luschgy \cite{GL1} showed that for any \( r > 0 \), the quantization dimension of order \( r \) exists for the invariant self-similar measure $\vartheta$ associated with an IFS satisfying the OSC, and is determined by the equation
\[
\sum_{i=1}^N (p_i c_i^r)^{\frac{D_r}{D_r +r}} =1.
\]
Observe that IFS $\mathcal{I}$ satisfies the SSC and hence the OSC, the quantization dimension $D_r$ of the self-similar measure $\vartheta$ corresponding to the IFS $\mathcal{I}$ with probability vector $\{p_1,p_2\}$ is determined by the following equation.
\[
\left[\frac{1}{3} \left(\frac{1}{3}\right)^r\right]^{\frac{D_r}{D_r + r}} + \left[\frac{2}{3} \left(\frac{1}{3}\right)^r\right]^{\frac{D_r}{D_r + r}} = 1.
\]
In particular, for \( r = 2 \), the quantization dimension of order 2 is approximately \( D_r \approx 0.6183 \).

Moreover, since \( \mathcal{I} \) satisfies the SSC, it follows from \cite[Theorem 7.4.1]{Fraser} that the lower dimension \( \dim_L \) of the self-similar measure \( \vartheta \) is given by
\[
\dim_L(\vartheta) = \min_i \left\{ \frac{\log p_i}{\log c_i} \right\}.
\]
Hence,
\[
\dim_L(\vartheta) = \frac{\log (2/3)}{\log (1/3)} \approx 0.3691.
\]

Additionally, since \( \mathcal{I} \) satisfies the open set condition (OSC), the Hausdorff dimension of \( \vartheta \) is given by $\dim_{\mathcal{H}}(\vartheta) = \frac{E_P}{L(P)}$ (see \cite{H}),
where \( E_P \) and \( L(P) \) denote the entropy and Lyapunov exponent of the self-similar measure \( \vartheta \), respectively, defined as $$E_P = -\sum_{i=1}^2 p_i \log p_i \quad \text{and} \quad L(P) = -\sum_{i=1}^2 p_i \log |c_i|.$$
Substituting the above values yields $\dim_{\mathcal{H}}(\vartheta) = \frac{(1/3)\log (1/3) + (2/3) \log (2/3)}{(1/3) \log (1/3) + (2/3) \log (1/3)} \approx 0.5794.$
\end{example}
\subsection{Lower dimension of invariant measure associated with product IFS} \label{subsec1}
In this subsection, we will see the relationship between the lower dimension of the invariant measure associated with the product of two IFSs and the lower dimension of the invariant measures associated with the individual IFSs. For more details about the product IFS, interested readers may see \cite{VP1}.

\begin{theorem}
    Let $\mathcal{I}_1:= \{\mathbb{R}^m; f_1,f_2,\ldots,f_N\},$ and $\mathcal{I}_2 := \{\mathbb{R}^m; g_1,g_2,\ldots,g_N\}$ be two IFSs consisting of similarity mappings with similarity ratio $c,$ for $1\le i \le N.$ Then, the lower dimension of $\vartheta := \vartheta_1 \times \vartheta_2$ is $s +t,$ where $s$ and $t$ are the lower dimensions of $\vartheta_1$ and $\vartheta_2$, respectively, provided that $\mathcal{I}_1$ and $\mathcal{I}_1$ satisfy the SSC.
\end{theorem}
\begin{proof}
    Since IFSs $\mathcal{I}_1$ and $\mathcal{I}_2$ satisfy the SSC, by \cite[Theorem 7.4.1]{Fraser}, we have  
\[
\dim_L(\vartheta_1) = \min_{i} \frac{\log p_i}{\log c_i} \quad \text{and} \quad \dim_L(\vartheta_2) = \min_i \frac{\log q_i}{\log c_i}.
\]  
Furthermore, if the IFSs $\mathcal{I}_1$ and $\mathcal{I}_2$ satisfy the SSC, then the product IFS also satisfies the SSC (see \cite[Theorem 3.23]{VP1}). Consider the probability vector $\{p_i q_j: 1\le i,j, \le N\}$; then, by \cite[Theorem 7.4.1]{Fraser}, we have  
\[
\begin{aligned}
    \dim_L(\vartheta) &= \min_{i,j} \frac{\log p_i q_j}{\log c} = \min_{i,j} \left( \frac{\log p_i}{\log c} + \frac{\log q_j}{\log c} \right) \\
    &= \min_i \frac{\log p_i}{\log c} + \min_i \frac{\log q_i}{\log c} 
    = s + t.
\end{aligned}
\]

Thus, we have $\dim_L(\vartheta) = s+t.$
\end{proof}
\begin{note}
In \cite[Theorem 6.6]{Sh1}, it is shown that if $\{f_i\}_{i=1}^N$ is an IFS satisfying the ESC, and $\vartheta$ is a self-similar measure associated with a probability vector $\{p_i\}_{i=1}^N$, then for $q \in (1, \infty)$, the $L^q$-dimension of $\vartheta$ satisfies 
$D(\vartheta, q) = \min\left\{\frac{\tilde{\tau}(q)}{1 - q},\, 1\right\}.$
Further, \cite[Remark 1.4]{KNZ1} states that for $q \in (0,1)$, we have
$\overline{D}_r(\vartheta) = \frac{D(\vartheta, q)}{q - 1}.$ However, combining these two results-Theorem 6.6 of \cite{Sh1} and Remark 1.4 of \cite{KNZ1}-does not yield the quantization dimension of $\vartheta$. 
\end{note}

\begin{op}
The quantization dimension of a self-similar measure $\vartheta$ associated with an IFS satisfying the ESC remains an open problem.
\end{op}
\begin{op}
    To the best of our knowledge, we do not find any literature that addresses the lower dimension of self-similar sets and measures on $\mathbb{R}$ under the ESC. Thus, it is still an open problem. 
\end{op}
   

\bibliographystyle{amsalpha}

\end{document}